\documentclass[11pt,reqno]{amsart}
\usepackage{amscd,amsmath,amsopn,amssymb,amsthm,multicol}
\usepackage{tikz,subdepth,anysize}
\usetikzlibrary{decorations.markings,arrows.meta,bending,calc}
\tikzset{>={Stealth[scale=1.5, bend]}}
\everymath=\expandafter{\the\everymath\displaystyle}
\usetikzlibrary{positioning}

\usetikzlibrary{arrows,decorations.markings,positioning}
\usetikzlibrary{matrix,decorations.pathmorphing}
\tikzset{inner sep=0pt, node distance=5mm,
  root/.style={circle,draw,minimum size=5pt,thick},
  broot/.style={circle,draw,minimum size=5pt,thick,fill},
  xroot/.style={circle,draw,minimum size=5pt,thick,label=below:$\times$},
  doublearrow/.style={postaction={decorate},   decoration={markings,mark=at position .6 with {\arrow[line width=1.2pt]{>}}},double distance=1.6pt,thick},
  rdoublearrow/.style={postaction={decorate},   decoration={markings,mark=at position .4 with {\arrowreversed[line width=1.2pt]{>}}},double distance=1.6pt,thick},
	rtriplearrow/.style={postaction={decorate},   decoration={markings,mark=at position .4 with {\arrowreversed[line width=1.2pt]{>}}},double distance=2.5pt,thick},
	ltriplearrow/.style={postaction={decorate},   decoration={markings,mark=at position .6 with {\arrow[line width=1.2pt]{>}}},double distance=2.5pt,thick},
  curvedline/.style={bend=right}
}

\newcommand\com[1]{}

\newcommand\D{{\mathcal D}}
\newcommand\E{\mathcal{E}}
\newcommand\g{{\frak g}}

\newcommand\m{{\frak m}}
\newcommand\fp{{\frak p}}
\newcommand\M{\hat{M}}
\newcommand\op[1]{\mathop{\rm #1}\nolimits}
\newcommand\p{\partial}
\renewcommand\P{{\mathbb P}}
\newcommand\R{{\mathbb R}}

\theoremstyle{plain}
\newtheorem{theorem}{Theorem}

\newtheorem{rk}{Remark}

\theoremstyle{definition}

\begin{document}


\title[Differential Invariants of Curves in $G_2/P$]%
{Differential invariants of curves\\ in ${\bold G_2}$ flag varieties}

\author{Boris Kruglikov}
\address{Department of Mathematics and Statistics, UiT the Arctic University of Norway, Troms\o\ 9037, Norway.
\ E-mail: {\tt boris.kruglikov@uit.no}. }

\author{Andreu Llabr\'{e}s}
\address{Department of Mathematics and Statistics, UiT the Arctic University of Norway, Troms\o\ 9037, Norway.
\ E-mail: {\tt andreu.llabres@uit.no}. }

 \begin{abstract}
We compute the algebra of differential invariants of unparametrized
curves in the homogeneous $G_2$ flag varieties, namely in $G_2/P$.
This gives a solution to the equivalence problem for such curves.
We consider the cases
of integral and generic curves and relate the equivalence problems
for all three choices of the parabolic subgroup $P$.
 \end{abstract}

\maketitle

\section*{Introduction}

Computation of differential invariants of (unparametrized) curves
was an important topic in XIX-th century mathematics.
Besides Frenet-Serret formulas for Euclidean spaces, curves
were extensively studied in projective spaces.
In particular, Klein and Lie derived special ODEs describing projective invariant classes of curves in the plane \cite{KlL} and Halphen
computed the invariants that govern projective equivalence \cite{H},
see \cite{KoL} for a modern approach and review.
Wilczynski \cite{W} derived a
complete set of invariants for such curves in any dimension,
starting from the Laguerre-Forsyth normal form for linear ODEs.

Invariants of curves were derived in various non-metric geometries,
for instance in conformal, Grassmannian and others \cite{B,F,G}.
Fundamental invariants for some classes of curves in generalized flag varieties
were computed by Doubrov and Zelenko \cite{DZ}. This approach was based on
the moving frame method and the theory of parabolic geometries.
In particular, it covered general curves in projective spaces
and integral curves in the $G=G_2$ flag varieties for split algebraic
$G_2=G_2^*$, namely homogeneous space $G/P$ with some parabolic subgroup $P$.

The goal of this paper is to revisit the case of curves in $M=G_2/P$
based on the theory of differential invariants for which we refer
the reader to \cite{KL}
(we assume the curves are regular, so that its velocity never vanishes).
The algebra of such (absolute rational) invariants depends on the type of 1-jet of a curve, and we compute the Hilbert function counting
the number of differential invariants for every type.
Then we concentrate on the two poles, corresponding to the minimal
and the maximal orbits of the action of $G$ on $J^1(M,1)=\mathbb{P}TM$.
For integral curves, representing minimal orbits, we obtain
the invariants differently and more explicitly than in \cite{DZ}.
For generic curves, representing maximal orbits, our results are
apparently new.

We give more details to the case $M=G_2/P_1$ (labelling of the
parabolic subgroups corresponds to the Bourbaki numeration),
where our tool is the canonical conformal structure associated to
the (2,3,5) distribution by Nurowski. In the flat case that we discuss,
it corresponds to another homogeneous representation
$M=SO(3,4)/P_1$ 
(here $P_1$ is the subgroup of $B_3$ naturally extending the previous $P_1$)
induced by the inclusion $G_2\subset SO(3,4)$.
We also discuss differential invariants of curves in other
$G_2$ flag varieties with a different choice of the parabolic, namely  $G_2/P_{12}$ and $G_2/P_2$, but since the formulae are large we
do not provide full details there.

It turns out that there is a transformation (multiple-valued to one side)
between generic curves in $G_2/P_1$ and $G_2/P_2$ based on the twistor
correspondence via $G_2/P_{12}$ and the geometry of those spaces.
For minimal integral curves (no restriction for $G_2/P_1$ but
corresponding to triple root Petrov type in the case $G_2/P_2$)
such a transformation is known as the prolongation,
but it does not exist for all curves. The transformation we propose
for generic curves (it is $1:1$ to one side and $1:2$ to the other)
allows to reduce computation of invariants to those of $G_2/P_1$.

On historical side we note that two realisations of $G_2$ corresponding
to homogeneous spaces $G_2/P_1$ and $G_2/P_2$ were obtained in 1893 by
E.~Cartan \cite{C} and F.~Engel \cite{E}. An interplay between these
models is a base for the above mentioned correspondence.

We also note that the method developed in this paper is applicable to
obtain differential invariants of curves in curved geometries of type $(G_2,P)$, as we will briefly discuss in the conclusion,
though we do not provide any explicit formulae for those.
 It would be interesting to extend the results
 to parabolic geometries of general type $(G,P)$.

\medskip

Maple computations supporting the results of this paper, and
containing some large formulae, can be found in a supplement
to the arXiv submission.

\bigskip

\textsc{Acknowledgment.}
The research leading to these results has received funding from the Norwegian Financial Mechanism 2014-2021 (project registration number 2019/34/H/ST1/00636), the Polish National Science Centre (NCN) (grant number 2018/29/B/ST1/02583), and the Troms\o{} Research Foundation (project ``Pure Mathematics in Norway'').

\section{Invariants of curves in $M^5=G_2/P_1$}\label{S1}

Associated with $P_1$ is the gradation of $\g=\op{Lie}(G_2)$ of depth 3
 \begin{equation}\label{g2p1}
\g=\g_{-3}\oplus\g_{-2}\oplus\g_{-1}\oplus\g_0\oplus\g_1\oplus\g_2\oplus\g_3
 \end{equation}
with $\dim\g_{\pm1}=\dim\g_{\pm3}=2$, $\dim\g_{\pm2}=1$ and
$\g_0=\mathfrak{gl}_2$. The filtration $\g^i=\oplus_{j\ge i}\g_j$
is invariant with respect to $\fp=\g^0$ and its Lie group
$P_1= GL_2\ltimes\exp(\fp_+)$, which is equal to $\op{Stab}_o(G_2)$,
$o=[P_1]$, for the $G_2$ action on the 5-dimensional homogeneous space $M=G_2/P_1$.

We identify $T_oM$, as well as tangent spaces at other points of $M$,
with $\m=\g/\fp$ and it will be convenient to interpret it as
$\g_-=\g_{-3}\oplus\g_{-2}\oplus\g_{-1}$ though the gradation is not
$P_1$ invariant. Furthermore $\g^{-1}\op{mod}\fp$ defines a $G_2$ invariant rank 2 distribution $\Pi$ on $M$ with growth vector $(2,3,5)$, which
exhibits ranks of the derived distributions $\Pi^2=[\Pi,\Pi]$ and $\Pi^3=[\Pi,\Pi^2]=TM$.

We use the coordinates $(x,y,p,q,z)$ on $M$ obtained from the
model Monge equation $y''=(z')^2$ with local $G_2$ symmetry.
Namely, denoting $y'=p$ and $y''=q$ the distribution $\Pi$
is induced from the Cartan distribution in jets and has the following expression:
 \begin{equation}\label{Pi}
\Pi=\langle \p_x+p\p_y+q\p_p+q^2\p_z,\p_q\rangle.
 \end{equation}
The $G_2$ invariant conformal structure $[g]$ is given by the representative
 \begin{equation}\label{conf}
g=q^2dx^2-2q\,dx\,dp+6p\,dx\,dq-3\,dx\,dz-6\,dy\,dq+4\,dp^2.
 \end{equation}

Let $J^k(M,1)$ denote the space of $k$-jets of unparametrized regular
curves (for details on jet-spaces we refer to \cite{KL0}).
These will be represented as parametrized curves $\gamma:I\to M$
modulo the right action of the pseudogroup $\op{Diff}_\text{loc}(\R)$
of reparametrizations.
A 0-jet is just $\gamma(0)\in M$, while 1-jet is $\dot\gamma(0)\neq0$
up to rescaling. Thus $J^0(M,1)=M$
and $J^1(M,1)=\mathbb{P}TM$. The fiber of the latter bundle over $o\in M$
will be identified with $(T_oM\backslash\{0\})/\R_\times\simeq\mathbb{P}\m$,
on which $P_1$ acts.

A general curve, transversal to the foliation $\{x=\op{const}\}$,
can be parametrized as $y=y(x)$, $p=p(x)$, $q=q(x)$, $z=z(x)$.
This introduces an affine chart in $J^k(M,1)$ with coordinates
$x,y,p,q,z$ and $y_i,p_i,q_i,z_i$ for $1\leq i\leq k$.
The action of $G_2$ will be expressed in these coordinates, as well as
the invariants of the action. We note that the orbits of $G_2$
in $J^k(M,1)$ are bijective with the orbits of $P_1$ on $J^k_o(M,1)$,
and we begin with a discussion of $k=1$ case.

\subsection{Action and orbits of $P_1$ on 1-jet}
It is convenient to describe the action on the vector space $\m$,
and then pass to the corresponding projective space (here 
when describing the $P$ action we use linear coordinates on $\m$,
and then return to the coordinates on $M$ used before).

A basis $e_1,e_2$ of $\g_{-1}$ induces bases $e_3=[e_1,e_2]$ of
$\g_{-2}$ and $e_4=[e_1,e_3]$, $e_5=[e_2,e_3]$ of $\g_{-3}$.
Thus we get a basis of $\m\simeq\oplus_{i=-3}^{-1}\g_i$.
This introduces coordinates $v_i$ on $\m$ considered as a
graded nilpotent Lie algebra
 \[
v= v_1e_1+v_2e_2+v_3e_3+v_4e_4+v_5e_5\in \mathfrak{m}\,,
 \]
and hence also on $\exp(\m)\simeq M$ via the exponential map. Note that
the unity $o$ corresponds to 0 and the inverse of $g\in\exp(\m)$ is $-g$.
(The Lie bracket on $\m$ induces the group structure on $\exp(\m)$ by the Baker-Campbell-Hausdorff formula that is finite
due to nilpotency of $\m$.)

In these coordinates the canonical conformal structure \eqref{conf}
has constant coefficients:
 \begin{equation}\label{conf2}
\langle v,w\rangle = v_1w_5+v_5w_1-v_2w_4-v_4w_2+v_3w_3.
 \end{equation}
Indeed, this is a unique (up to scale) left-invariant metric on $M$ that
is also invariant with respect to semi-simple part of the reductive
component $GL_2$ of $P_1$ (this preserves the grading on $\m$) and
has weight 4 with respect to its center:
 \[
\langle\op{Ad}_{g}(v),\op{Ad}_{g}(w)\rangle=\langle v,w\rangle
\quad\forall\,g\in\exp(\m)
 \quad\text{ and }\quad
\langle Av,Aw\rangle=(\det A)^2\langle v,w\rangle\quad \forall\,A\in GL_2\,.
 \]

To obtain $P_1$ action on $\m=\g/\fp$ we derive first the action of
$\fp_+$ using the root diagram of $\g$. We choose the basis $f_i$
of $\fp_+\simeq\m^*$ dual to $e_i$\,: these are the root vectors
indicated near the corresponding roots on the diagram below
(gradation is shown by a family of parallel lines). Their commutator
relations are $[f_1,f_2]=f_3$, $[f_1,f_3]=f_4$, $[f_2,f_3]=f_5$.

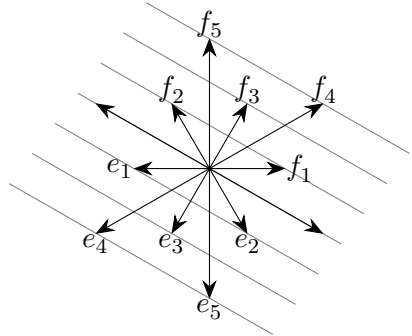
\begin{center}
\begin{tikzpicture}
\draw[gray, xshift=9mm, yshift=12mm] (150:2) to (150:-2);
\draw[gray, xshift=3mm, yshift=4mm] (150:2) to (150:-2);
\draw[gray, xshift=6mm, yshift=8mm] (150:2) to (150:-2);
\draw[gray] (150:2) to (150:-2);
\draw[gray, xshift=-3mm, yshift=-4mm] (150:2) to (150:-2);
\draw[gray, xshift=-6mm, yshift=-8mm] (150:2) to (150:-2);
\draw[gray, xshift=-9mm, yshift=-12mm] (150:2) to (150:-2);
\draw[<->] (0:1)node[right]{$f_1$} to (0:-1) node[left]{$e_1$};
\draw[<->] (120:1)node[above]{$f_2$} to (120:-1)node[below]{$e_2$};
\draw[<->] (60:1)node[above]{$f_3$} to (60:-1)node[below]{$e_3$};
\draw[<->] (30:{sqrt(3)})node[above]{$f_4$} to (30:-{sqrt(3)}) node[below]{$e_4$};
\draw[<->] (90:{sqrt(3)})node[above]{$f_5$} to (90:-{sqrt(3)}) node[below]{$e_5$};
\draw[<->] (150:{sqrt(3)}) to (150:-{sqrt(3)});
\end{tikzpicture}
\end{center}

Brackets between basis elements in $ \mathfrak{p}_+ $ and $ \mathfrak{m} $ correspond, up to scale, to vector summation of the corresponding root vectors in the root diagram: $[e_\alpha,e_\beta]=k_{\alpha\beta}e_{\alpha+\beta}$.
The factors $k_{\alpha\beta}$ can be fixed by the condition
that the metric is invariant under $\fp_+$:
 \[
\langle\op{ad}_{f_i}v,w\rangle+\langle v,\op{ad}_{f_i}w\rangle=0.
 \]
The diagram implies that $\g_3=\langle f_4,f_5\rangle$ acts trivially on
$\m$ and $\fp_+$ acts trivially on $\g_{-1}$. The remaining commutation relations are as follows.
 \[
\begin{array}{lll}
[f_1,e_3] = e_2, & [f_2,e_3] = -e_1, & [f_3,e_3] = 0,\\[2mm]
[f_1,e_4] = e_3, & [f_2,e_4] = 0, & [f_3,e_4] = e_1,\\[2mm]
[f_1,e_5] = 0, & [f_2,e_5] = e_3, & [f_3,e_5] = e_2.
\end{array}
 \]

Thus the action of $\exp(\fp_+)$ is encoded through the action of
$\rho=\exp(s_1f_1+s_2f_2+s_3f_3)$ with real parameters
$s_1,s_2,s_3$ as follows:
 \begin{multline*}
\rho(v_1,v_2,v_3,v_4,v_5)=\hspace{17pt}\\
\hspace{17pt}\bigl(v_1-s_2v_3+(s_3-s_1s_2)v_4-\tfrac12s_2^2v_5,
v_2+s_1v_3+\tfrac12s_1^2v_4+(s_3+s_1s_2)v_5,
v_3+s_1v_4+s_2v_5,v_4,v_5\bigr).
 \end{multline*}

With these at hand we can now compute that,
under the action of $P_1$, a general point
$v_1e_1+v_2e_2+v_3e_3+v_4e_4+v_5e_5\in\m\setminus\{0\}$ can be mapped
to one of the following 5 representatives:
 \begin{itemize}
\item $e_5\pm e_1$ if $(v_4,v_5)\neq(0,0)$ and $2v_1v_5-2v_2v_4+v_3^2\gtrless0$,
\item $e_5$ if $(v_4,v_5)\neq(0,0)$ and $2v_1v_5-2v_2v_4+v_3^2=0$,
\item $e_3$ if $v_4=v_5=0$ and $v_3\neq0$,
\item $e_1$ if $v_3=v_4=v_5=0$.
 \end{itemize}

Associated to the conformal structure \eqref{conf2} is the null cone
 \[
N= \left\{v_1e_1+v_2e_2+v_3e_3+v_4e_4+v_5e_5 \in \mathfrak{m}\,\,\middle|\,\, 2v_1v_5-2v_2v_4+v_3^2=0\right\}\,.
 \]
Its relation to the distributions $\Pi$ and $\Pi^2$ is the following:
 \[
N\cap\Pi^2 = \Pi\,.
 \]

Each of the 5 points above and the singular orbit 0 represent orbits
of the action of $P_1$ on $\m$. Therefore there are 5 orbits of
the action of $P_1$ on $\mathbb{P}\m$.

\begin{center}
\begin{tikzpicture}[scale=0.5]
\draw[left color=gray!50, right color=gray!50, middle color=white] (3.88,-3.88) to (0,0) to (-3.88,-3.88) arc (165:375:4.02 and 1);
\draw[fill=gray!50] (-4.45, 3.5) to (4, -4.95) to (4.1, -3.2) to (-4.1, 4.95) to cycle;
\draw[left color=gray!50, right color=gray!50, middle color=white] (3.82,3.82) to (0,0) to (-3.82,3.82);
\draw[bottom color=gray!60, top color=white] (0,4.01) ellipse (3.88 and 1);
\draw[dashed] (-3.82,-3.82) arc (165:15:3.88 and 0.8);
\draw[black!80, line width=0.5mm] (-4.3,4.3) to (4,-4);
\draw[densely dotted] (0,3) to (0,6) (0,-5.2) to (0,-6);
\draw[fill] (0,4.5) circle (0.2);
\draw[fill] (4.5,0) circle (0.2);
\draw[fill=white] (-3,-3) circle (0.2);
\draw[fill=white] (-3,2.5) circle (0.2);
\draw[fill=white] (3,-3) circle (0.2);
\draw[fill=white] (0,0) circle (0.2);
\node at (-4.7,-4){$ N $};
\node at (-4.5,2.6){$ \Pi^2 $};
\node at (-4.7,4.3){$\Pi $};
\node[fill=white] at (-4.7,0){$ \mathfrak{m}$};
\end{tikzpicture}
\end{center}

The points $e_5\pm e_1$ represent 2 open orbits, separated by the null cone
(both orbits connected, the above 3D picture is just an analogy).
The other orbits are (projectivizations of)
$\Pi\setminus\{0\}\ni e_1$, $\Pi^2\setminus\Pi\ni e_3$,
$N\setminus\Pi\ni e_5$.

\subsection{Number of invariants}\label{S12}

Define $s_k$ to be the transcendence degree of the field of
rational differential invariants of the $G_2$ action on $J^k(M,1)$.
By \cite{KL} this value is equal to codimension of the regular orbit
of the action, and so $s_k$ can be interpreted as the number of
invariants of order $k$. By the preceeding computations, $s_0=s_1=0$.

Hilbert function is defined as $h_k=s_k-s_{k-1}$ and it can be interpreted
as the number of invariants of pure order $k$.
We will consider curves of a fixed type $\texttt{t}$ of their 1-jet.
In other words, we assume that at any point 1-jet of the curve belongs
to the same $P_1$ orbit in $\mathbb{P}\m$.

If $d_\texttt{t}$ is dimension of the orbit, then such curves are
given by $4-d_\texttt{t}$ equations of the first order.
Geometrically this specifies a submanifold $\E_\texttt{t}\subset J^1(M,1)$
of codimension $4-d_\texttt{t}$, and prolongations
$\E^k_\texttt{t}:=\E^{(k-1)}_\texttt{t}\subset J^k(M,1)$ form a tower
of bundles with the rank of $\pi_{k,k-1}:\E^{k}_\texttt{t}\to\E^{k-1}_\texttt{t}$ equal
to $d_\texttt{t}$.
Because $G_2$ is finite-dimensional, we will occasionally have $h_k=d_\texttt{t}$ for $k\gg1$.

Let us consider the case of open orbits in $J^1(M,1)$.
Even though there are two types of such, the count is the same because
under complexification these become one orbit and the action is algebraic.
Due to transitivity we can fix 0-jet to be $a_0=(0,0,0,0,0)$
in $(x,y,p,q,z)$ coordinates. A generic 1-jet $a_1$ over $a_0$
can be represented by $(y_1,p_1,q_1,z_1)=(0,0,0,-1)$;
note that $z_1=+1$ for the other open orbit but we will focus on the first.

Denote by $\fp^{(k)}$ the prolongation of the isotropy $\fp$ of
the point $o=a_0$ in $\g$. The action in the fibers of $\pi_{k,k-1}$ for
$k>1$ is affine, so we can restrict to usage of Lie algebras (not groups).
Denote the isotropy subalgebra at the point $a_1$ by
 $$
\op{stab}_{a_1}=\{v\in\fp^{(1)}\,:\,v(a_1)=0\}
 $$
and by $\op{stab}_{a_1}^k$ its prolongations to $k$-jets.
As an abstract Lie algebra, $\op{stab}({a_1})$ is solvable and
it is defined by the following structure relations
 \begin{gather*}
[s_1,s_2]=s_2,\quad [s_1,s_3]=s_3,\quad [s_1,s_4]=s_4,\\
[s_1,s_5]=2s_5,\quad [s_2,s_4]=3s_5,\quad [s_3,s_4]=4s_5.
 \end{gather*}

The action of $\op{stab}_{a_1}^2$ on $\pi_{2,1}^{-1}(a_1)\simeq\R^4(y_2,p_2,q_2,z_2)$
is generated by the following vector fields
 \[
\p_{p_2}-z_2\p_{q_2}+4y_2\p_{z_2},\quad
\p_{q_2},\quad \p_{z_2},\quad p_2\p_{p_2}+2q_2\p_{q_2}+z_2\p_{z_2},
\quad 3y_2\p_{p_2}+4p_2\p_{q_2}
 \]
and so the orbits have dimension 3.
We have one invariant $I_2\!\!'=y_2$ in $\pi_{2,1}^{-1}(a_1)$.
Let $a_2=(I_2\!\!',0,0,0)\in J^2(M,1)$ be a point above $a_1$
(the expression for $I_2\!\!'$ is invariant only above $a_1$).

The isotropy algebra of $a_2$ is a 2-dimensional solvable subalgebra
$\op{stab}_{a_2}\subset\fp^{(2)}$. Its prolongation to 3-jets,
namely to $\pi_{3,2}^{-1}(a_2)\simeq\R(y_3,p_3,q_3,z_3)$, has generators
 \[
3y_3\partial_{p_3} + (3y_2z_3 + 4p_3)\partial_{q_3} - 12y_2y_3\partial_{z_3}\,,\quad
y_3\partial_{y_3} + 2p_3\partial_{p_3} + 3q_3\partial_{q_3} + 2z_3\partial_{z_3}\,,
 \]
and so we get 2 independent invariants in $\pi_{3,2}^{-1}(a_2)$.
The isotropy algebra of a generic point in 3-jet is already trivial,
so there will be 4 more independent invariants for every jet of order $k\ge4$.

The count of invariants for other types $\texttt{t}$ of 1-jet
is performed similarly, so we omit the details.
Summarizing, the Hilbert function $h_k$ counting
differential invariants is given in the table:

\[
\begin{array}{|c|cccccccccccl|}
\hline
\texttt{t}\quad \diagdown\quad k & 0 & 1 & 2 & 3 & 4 & 5 & 6 & 7 & 8 & 9 & 10 & \dots\\\hline
TM\backslash (N \cup \Pi^2)& 0 & 0 & 1 & 2 & 4 & 4 & 4 & 4 & 4 & 4 & 4 & 4 \\
N\backslash\Pi^2 & 0 & 0 & 0 & 1 & 2 & 3 & 3 & 3 & 3 & 3 & 3 & 3\\
\Pi^2\backslash\Pi & 0 & 0 & 0 & 0 & 0 & 1 & 2 & 2 & 2 & 2 & 2 & 2\\
\Pi\backslash \left\{ 0 \right\}  & 0 & 0 & 0 & 0 & 0 & 0 & 0 & 0 & 0 & 0 & 1 & 1\\\hline
\end{array}
\]

\medskip

Let us observe from the first two rows that for the corresponding
$\texttt{t}$-types there is dimensional freedom for the group
to act freely on the level of 2- and 3-jets, respectively, yet there
appear invariants. The situation with the last two rows is what could
be expected in a general position.

In what follows we will focus on two particular cases: curves with minimal
and maximal $\texttt{t}$-types of 1-jet (last and first rows) for
which we describe the algebra of differential invariants explicitly.
These algebras of differential invariants will be denoted by
$\mathcal{A}_\imath$ and $\mathcal{A}_g$ respectively.

\subsection{Invariants of integral curves}
\label{section_integral_curves1}

Consider curves in $M$ tangent to $\Pi$. For such there are no absolute differential invariants up to jet-order 9, and the first invariant
arises in order 10. There are however relative differential invariants:
such functions $R$ of order $k$ satisfy $L_vR=\alpha(v)R$ for
$v\in\g^{(\infty)}$ and $\alpha\in C^\infty(J^k)\otimes\g^*$.
We find those by the method of Sophus Lie, namely via a computation
of the loci where the rank of prolongations $e_j^{(k)}\in\mathfrak{X}(J^k)$
drop, for a basis $e_j\in\g$.

The integral curves are subject to the constraints
 \begin{equation}\label{IC1}
y_1=p,\ p_1=q,\ z_1=q^2,
 \end{equation}
their prolongation define the equation $\E_\Pi\subset J^\infty(M,1)$
with coordinates $x,y,p,q,z,q_k$, $k\ge1$.

The simplest relative differential invariant is $q_2$.
Next, such invariant arises in order 8, namely
 \begin{multline*}
R_8 = 196\,q_2^5q_8-2352\,q_2^4q_3q_7-5040\,q_2^4q_4q_6
-3255\,q_2^4q_5^2+16632\,q_2^3q_3^2q_6+59598\,q_2^3q_3q_4q_5\\
+13772\,q_2^3q_4^3-83160\,q_2^2q_3^3q_5-174735\,q_2^2q_3^2q_4^2
+297000\,q_2q_3^4q_4-118800\,q_3^6.
 \end{multline*}

One more relative invariant of order 10 is given by the formula
 \[
R_{10} = 21q_2R_8\mathcal{D}_x\bigl(q_2\mathcal{D}_xR_8\bigr) -
\frac{91}{4}\bigl(q_2\mathcal{D}_xR_8\bigr)^2 + 9R_8^2\bigl(13q_3^2-19q_2q_4\bigr).
 \]

The two latter relative invariants have proportional weights (that is, 1-forms
$\alpha$), which makes a combination of them an absolute
differential invariant
 \[
I_{10} = \frac{R_{10}^3}{R_8^7}.
 \]

This is accompanied by an invariant derivation, i.e. a linear map
$\Box:C^\infty(J^k)\to C^\infty(J^{k+1})$ that satisfies the
Leibniz rule and commutes with the action of $G_2$.
We search for it in the form $\Box=h\mathcal{D}_x$,
$h\in C^\infty(J^k)$, where
 \[
\D_x=\p_x+p\p_y+q\p_p+q^2\p_z+\sum_{i=0}^\infty q_{i+1}\p_{q_i}
 \]
is the operator of total derivative on $\E_\Pi$. Then the invariance condition
 \[
[v^{(\infty)},\Box]=0\quad\forall v\in\g
 \]
applied to $x$ writes for a basis $e_j\in\g$
 \[
L_{e_j^{(k)}}(h)=h\mathcal{D}_x(e_j(x)).
 \]
Finding $h$ from this equation, we get the following invariant derivation
($\imath$ for `{\i}ntegral'):
 \[
\Box_\imath= \frac{q_2}{R_8^{1/6}}\mathcal{D}_x.
 \]
The invariant derivation $\Box_\imath$ produces the next differential invariant
$I_{11}=\Box_\imath(I_{10})$ of order 11, then
$I_{12}=\Box_\imath(I_{11})$ of order 12, and successively generates
all the higher order invariants.

 \begin{theorem}\label{thm1}
The algebra $\mathcal{A}_\imath$ of (micro-local) differential invariants
of integral curves is generated in the Lie-Tresse sense by $I_{10}$ and \,$\Box_\imath$.
 \end{theorem}

 \begin{rk}\rm\label{rk1}
The invariant derivation $\Box_\imath$ has non-rational coefficient. This is
sufficient for micro-local invariants (defined in open non-invariant
sets in jets), however is at odd with the claim
that the global differential invariants are rational in jets \cite{KL}.
To remedy this one passes to rational invariant derivation
(Tresse derivative associated to $I_{10}$)
 $$
\bar\Box_\imath=\frac{d}{dI_{10}}:=\frac{I_{10}}{\Box_\imath(I_{10})}\cdot\Box_\imath.
 $$
Then the algebra $\bar{\mathcal{A}}_\imath$ of global differential invariants
for integral curves is generated by $I_{10}$, $\bar{I}_{11}=I_{11}^6$
and $\bar\Box_\imath$. Note that $\bar\Box_\imath(I_{10})=1$.
 \end{rk}

 \begin{proof}
By the count of invariants there is precisely one independent differential invariant
of pure order $k$ for any $k\ge10$. These are $\Box_\imath^{k-10}I_{10}$.
The micro-local claim follows.

To obtain the algebra of rational
differential invariants let us note that the invariants $\bar\Box_\imath^{k-11}\bar{I}_{11}$
are affine in jets of order $k>11$. Invariants of order $\leq11$ are
algebraically generated by $I_{10}$ and $\bar{I}_{11}$ since the ideal
generated by them (in the ring of rational functions that are polynomial
in jets of order $>9$) is radical. This proves the claim.
 \end{proof}

\subsection{Invariants of generic curves}\label{Sec_gen1}

Now we consider curves transversal to the distribution and not null
with respect to conformal structure \eqref{conf}. Investigation
of curves of both general type $\texttt{t}$ of 1-jets goes parallel,
so we may assume that the tangent $X=\dot\gamma$ to the curve
satisfies $g(X,X)>0$ (one has to take another normalization below for
$g(X,X)<0$).

When the curve is parametrized coordinately
$\gamma(t)=(x(t),y(t),p(t),q(t),z(t))$, 
its tangent vector
 \begin{equation}\label{XX}
X = \p_x+y_1\p_y+p_1\p_p+q_1\p_q+z_1\p_z
 \end{equation}
is given by the truncated total derivative.
Recall it is defined up to scale. We are going to exploit the change
of scales in order to construct an invariant frame along the curve.

The metric $g$ in \eqref{conf} is defined up to rescaling.
Another representative of $[g]$ is given by $\bar{g}=e^{2f}g$.
Let $\nabla,\bar\nabla$ be the Levi-Civita connections of $g,\bar{g}$.
They are related as follows:
 \begin{equation}\label{nabla}
\overline{\nabla}_XU= \nabla_XU+X(f)U+U(f)X-g(X,U)\nabla{f}
 \end{equation}
for $X,U\in\mathfrak{X}(M)$.
Therefore $\nabla_XU$ is defined up to $X$, $U$ and $\nabla{f}$. The
latter is difficult to control, so we will apply this formula only for
$g(X,U)=0$. Also, to obtain invariant quantities we can only differentiate
in the direction of the curve, so $X$ will be taken as in \eqref{XX}.

Since $\Pi^2/\Pi$ has rank 1, there is a conformal identification
$\Pi\simeq TM/\Pi^2$ based on $\g_{-1}\simeq[\g_{-1},\g_{-2}]=\g_{-3}$ of
\eqref{g2p1}. Since $\gamma$ is a generic curve, its tangent
$X\notin\Pi^2$ has a conformal dual $Y\in\Pi$.
For instance, choosing the vector
$e_3=[\p_q,\p_x+p\p_y+q\p_p+q^2\p_z]=\p_p+2q\p_z\in\Pi^2$
we find a unique $Y\in\Pi$ from
 \[
[Y,e_3] = X\,\op{mod}\Pi^2.
 \]
This $Y$ is defined up to scale and satisfies $g(X,Y)=0$.
Therefore the covariant derivative $\nabla_XY$ is defined up to $X,Y$
and determines unambiguously the subspace
 \[
\Pi_X= \langle X,Y,\nabla_XY\rangle\subset TM.
 \]
Generically $\Pi_X$ has rank 3, $\Pi_X\operatorname{mod}\Pi$ rank 2,
and $\Pi_X\operatorname{mod}\Pi^2$ rank 1.
We change the generator $\nabla_XY$ of $\Pi_X$ to
 \[
Z\in \Pi_X\cap\Pi^2\quad\text{ such that }\quad g(X,Z-X)=0.
 \]
This $Z$ is defined up to $Y$ and up to scale. However,
due to the above relation, $X$ and $Z$ are subject to rescaling
by the same factor. Therefore we have
 \[
\Pi_X=\langle X, Y, Z\rangle
 \]
and we use $Z$ to fix the scale of $Y$: the conformal identification
$[Y,Z]=X\,\op{mod}\Pi^2$ determines $Y$ uniquely.

The Gram matrix of $Y,Z,X$ (in this order) is
 \[
\left( \begin{array}{ccc}
0 & 0 & 0 \\
0 & \varkappa_2 & \varkappa_1 \\
0 & \varkappa_1 & \varkappa_1
\end{array} \right)
 \]
where $\varkappa_1=2R_1$, $\varkappa_2=\frac{R_1^4}{162R_2^2}$ and
 \begin{equation}\label{R1}
\begin{array}{l}
R_1 = (q+2p_1)^2 + 6(q_1(p-y_1)-qp_1) - 3z_1\,,\\[2mm]
R_2 =-\frac{(q + 2p_1)^3}{18} +(p-y_1)\left(qp_2-\frac{z_2}{2}\right) + (q+y_2)\left(qp_1-\frac{z_1}{2}\right) + p_1z_1 - q^2 \frac{y_2}{2}
\end{array}
 \end{equation}
are relative invariants of orders 1 and 2 respectively.
Their ratio
 \begin{equation}\label{I2}
I_2=\frac{\varkappa_1}{324\varkappa_2}=\frac{R_2^2}{R_1^3}.
 \end{equation}
is the first absolute differential invariant in $J^2(M,1)$;
note that this invariant, when restricted to $\pi_{2,1}^{-1}(a_1)$,
differs from $I_2\!\!'$ of subsection \ref{S12} only by a power and a factor:
$108I_2=(I_2\!\!')^2$.

As long as $\varkappa_1\neq \varkappa_2$, which is generically true, the conformal metric
has rank 2 on the 3-dimensional $\Pi_X$.
Hence $\Pi_X^\perp$ has rank 2 and $\Pi_X\cap\Pi_X^\perp=\langle Y\rangle$.
Choose
 \[
V\in\Pi\quad\text{ such that }\quad g(X,V-X)=0.
 \]
This $V$ is defined up to $Y$ and up to scale with the same factor as $X$.
Therefore we have
 \[
\Pi_X+\Pi_X^\perp = \langle V,Y,Z,X\rangle\,.
 \]
To complete to a (yet non-canonical) frame we add the vector
 \[
W= \nabla_X(Z-X).
 \]
Since $g(X,Z-X)=0$, this $W$ defined up to scale, and up to $X$ and $Z$, which in turn is defined up to $X$ and $Y$. The vectors $X$ and $W$ are independent mod $\Pi^2$.

The Gram matrix of $Y,V,Z,X,W$ (in this order) is
 \[
\left( \begin{array}{ccccc}
0 & 0 & 0 & 0 & \varkappa_3 \\
0 & 0 & 0 & \varkappa_1 & k_4 \\
0 & 0 & \varkappa_2 & \varkappa_1 & k_3 \\
0 & \varkappa_1 & \varkappa_1 & \varkappa_1 & k_2 \\
\varkappa_3 & k_4 & k_3 & k_2 & k_1
\end{array} \right)
 \]
where $\varkappa_3=2R_1\Bigl(108I_2-\frac13\Bigr)$ is a relative invariant.

Keeping track of the choices, we change our vectors and the Gram matrix changes accordingly. Precisely, the freedom we have in defining our vectors allows us to make the following transformations:
 \[
\begin{array}{lll}
X & \mapsto & c_1X\,,\\[2mm]
Y & \mapsto & \phantom{c_1}Y\,,\\[2mm]
Z & \mapsto & c_1Z + c_2Y\,,\\[2mm]
V & \mapsto & c_1V + c_3Y\,,\\[2mm]
W &\mapsto & c_1^2W + c_4Z -(c_4+c_1c_2k_{5})X + c_5Y\,,
\end{array}
 \]
where $k_5 = \frac{p_1-q}{3}+\frac{3R_2}{R_1}$ and we have 5 degrees of freedom given by $c_1,c_2,c_3,c_4,c_5$. We find $c_2,c_3,c_4,c_5$
such that $k_1=k_2=k_3=k_4=0$. Under this transformation of our vectors,
the Gram matrix of $Y,V,Z,X,W$ takes the form
 \[
\left( \begin{array}{ccccc}
0 & 0 & 0 & 0 & c_1^2\varkappa_3 \\
0 & 0 & 0 & c_1^2\varkappa_1 & 0\\
0 & 0 & c_1^2\varkappa_2 & c_1^2\varkappa_1 & 0\\
0 & c_1^2\varkappa_1 & c_1^2\varkappa_1 & c_1^2\varkappa_1 & 0\\
c_1^2\varkappa_3 & 0 & 0 & 0 & 0
\end{array} \right) \,.
 \]
It only remains to fix the scale of $X$ and the scale of the metric $g$.

We fix $c_1$ by the condition $\mathcal{L}_XI_2=1$, that is, by setting
 \[
c_1 = \frac{1}{\mathcal{D}_xI_2},
 \]
where $\mathcal{D}_x$ is the operator of total derivative:
 \[
\D_x=\p_x+\sum_{i=0}^\infty \Bigl(y_{i+1}\p_{y_i}+p_{i+1}\p_{p_i}+q_{i+1}\p_{q_i}+z_{i+1}\p_{z_i}\Bigr).
 \]
Then we fix the scale of the metric $\bar{g}=e^{2f}g$ by
 \begin{equation}\label{gg}
\bar{g}(X,X)=1.
 \end{equation}
This fixes an invariant frame adapted to the distribution and the 
conformal structure. With this the above Gram matrix
becomes
 \[
\left( \begin{array}{ccccc}
0 & 0 & 0 & 0 & 108I_2-\tfrac13 \\
0 & 0 & 0 & 1 & 0\\
0 & 0 & \tfrac1{324}I_2^{-1} & 1 & 0\\
0 & 1 & 1 & 1 & 0\\
108I_2-\tfrac13 & 0 & 0 & 0 & 0
\end{array} \right) \,.
 \]

The next step is to generate differential invariants of orders 3
and 4 in $J^4(M,1)$.
In order to do so first note that the new metric given by \eqref{gg}
involves rescaling depending on 3-jet, and so its Levi-Civita connection
is uncomputable (we can only differentiate along the curve), yet
formula \eqref{nabla} applied to $U\in\langle X\rangle^\perp$ has a
well-defined output $\bar{\nabla}_XU\,\op{mod}X$.

Let $w_i$ denote the basis $Y,Z-X,V-X,W$ of $\langle X\rangle^\perp$.
Then we decompose
 \[
\bar\nabla_Xw_i=\sum_{j=1}^4 a_{ij}w_j\,\op{mod}X,\quad 1\leq i\leq 4.
 \]
The coefficients $a_{ij}$ are differential invariants. Some of them
are constants or expressed through $I_2$, some are related due to the
fact that $\bar\nabla_X\bar{g}=0$, but the others will determine
2 differential invariants $I_{3a},I_{3b}$ of order 3
and 4 differential invariants $I_{4a},I_{4b},I_{4c},I_{4d}$
of order 4 (all rational and independent).
The formulae are too large to be given explicitly, but in
$\pi_{4,2}^{-1}(a_2)$ the invariants $I_{3a},I_{3b}$ are equal to
$\frac{4I_2p_3+z_3}{y_3^2}$, $\frac{16p_3^2-24y_3q_3-3z_3^2}{y_3^4}$,
while $I_{4j}$ are affine in $y_4,p_4,q_4,z_4$.

Finally the normalized $X$ yields the invariant derivation
(Tresse derivative associated to $I_2$)
 \[
\Box_g=\frac{d}{dI_2}:= \frac{1}{\mathcal{D}_xI_2}\cdot\mathcal{D}_x
 \]
which on scalars coincides with $L_X=\bar\nabla_X$.

 \begin{theorem}\label{Th2}
The algebra $\mathcal{A}_g$ of differential invariants of generic curves
is generated in the Lie-Tresse sense by seven differential invariants
$I_2,I_{3i},I_{4j}$ and one invariant derivation $\Box_g$.
 \end{theorem}

 \begin{proof}
By construction the seven invariants generate all differential invariants
of order $\leq4$. Invariant derivation provides independent invariants affine
in jets of order $k>4$ in the totality equal to the number of those jets.
Thus any differential invariant from $\mathcal{A}_g$ can be rationally expressed
through the given generators.
 \end{proof}

\section{Invariants of curves in $\M^6=G_2/P_{12}$}\label{S2}

Associated with $P_{12}$ is the gradation of $\g=\op{Lie}(G_2)$ of depth 5
 \begin{equation}\label{g2p12}
\g=\g_{-5}\oplus\g_{-4}\oplus\g_{-3}\oplus\g_{-2}\oplus\g_{-1}\oplus\g_0\oplus\g_1\oplus\g_2\oplus\g_3\oplus\g_4\oplus\g_5
 \end{equation}
with $\dim\g_{\pm1}=\dim\g_0=2$ and $\dim\g_i=1$ otherwise.
The filtration $\g^i$ is introduced as in Section \ref{S1} and it defines
the distribution $\Delta$ of growth $(2,3,4,5,6)$ invariant with
respect to $G_2$; the stabilizer of $o\in\hat M$ is
$P_{12}= (\R^1_\times\times\R^1_\times)\ltimes\exp(\fp_+)$.

To introduce coordinates on $\hat{M}$ it is convenient to identify it with
the prolongation of $(M,\Pi)$, namely the $\P^1$ bundle $\P\Pi$ over $M$:
its points are $\hat{a}=(a,\ell)$, where $a\in M$ is a point and $\ell\subset\Pi_x$
is a line. Thus we can use the coordinates $(x,y,p,q,z,r)$ for an open chart in $\hat{M}$,
where the first 5-tuple gives a chart in $M$ as in Section \ref{S1} and the line $\ell$
has coordinates $[1:r]$ in the basis \eqref{Pi} of $\Pi$. This gives the representation
 \begin{equation}\label{PPi}
\Delta=\hat{\Pi}=\langle\p_x+p\p_y+q\p_p+q^2\p_z+r\p_q,\p_r\rangle
 \end{equation}
where both generators are distinguished: the first by the prolongation procedure
described above and the second as the kernel of the differential of
the projection $\pi_l:\M\to M$.

This projection relates the derived distributions as follows:
$\pi_l^{-1}(\Pi)=\Delta^2$ and $\pi_l^{-1}(\Pi^2)=\Delta^3$. In addition,
the pullback of the conformal structure on $M$ gives a degenerate conformal structure on $\M$ with the null cone $\hat{N}=\pi_l^{-1}(N)\simeq N\times\R^1$ in $T\M$.

We parametrize the curves again by $x$, so the jet-coordinates on $J^\infty(\M,1)$
are $(x,y,p,q,z,r)$ and $(y_k,p_k,q_k,z_k,r_k)$. The null cone is given
by the condition $R_1=0$ of the $G_2/P_1$ case \eqref{R1}.
The equation for derived flag is given by the following conditions
 \[
\begin{array}{ll}
\E_\Delta \ \,= &\! \left\{ q_1 = r, p_1 = q, y_1=p, z_1 = q^2 \right\} \\[2mm]
\E_{\Delta^2} \,= &\! \left\{p_1 = q, y_1=p, z_1 = q^2 \right\} \\[2mm]
\E_{\Delta^3} \,= &\! \left\{y_1=p, z_1 = 2qp_1 - q^2 \right\} \\[2mm]
\E_{\Delta^4} \,= &\! \left\{z_1 = 2pr  -2ry_1 + 2qp_1 -q^2 \right\} \\[2mm]
\end{array}
 \]
that determine some types of 1-jets; curves with the given fixed type
are solutions to the corresponding prolonged first order systems
$\E_{\Delta^s}\subset J^\infty(\hat{M},1)$.
There are however more types of 1-jets of curves in $\hat{M}$
as we will describe next.

\subsection{Action and orbits of $P_{12}$ on 1-jets}

The generators of \eqref{PPi} correspond to a basis $e_1,e_2$ of
$\g_{-1}$ in \eqref{g2p12}, which generates a basis $\{e_i\}$ of $\m$ via
commutation:
 \begin{gather*}
[e_1,e_2]=e_3,\ [e_1,e_3]=e_4,\ [e_1,e_4]=e_5,\
[e_3,e_4]=e_6,\ [e_2,e_5]=-e_6,\\
[f_1,f_2]=f_3,\ [f_1,f_3]=f_4 ,\ [f_1,f_4]=f_5,\
[f_3,f_4]=f_6,\ [f_2,f_5]=-f_6.
 \end{gather*}

This and its dual basis $\{f_i\}$ of $\fp_+$ are indicated on
the root diagram as before.

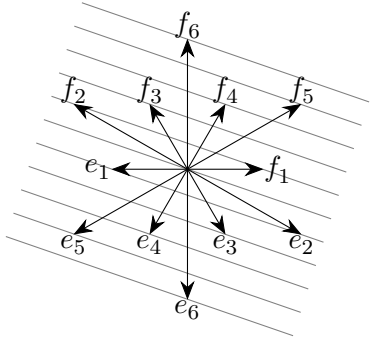
\begin{center}
\begin{tikzpicture}
\draw[gray, xshift=5mm, yshift=15.5mm] (160.8:2) to (160.8:-2);
\draw[gray, xshift=4mm, yshift=12.4mm] (160.8:2) to (160.8:-2);
\draw[gray, xshift=3mm, yshift=9.3mm] (160.8:2) to (160.8:-2);
\draw[gray, xshift=2mm, yshift=6.2mm] (160.8:2) to (160.8:-2);
\draw[gray, xshift=1mm, yshift=3.1mm] (160.8:2) to (160.8:-2);
\draw[gray] (160.8:2) to (160.8:-2);
\draw[gray, xshift=-1mm, yshift=-3.1mm] (160.8:2) to (160.8:-2);
\draw[gray, xshift=-2mm, yshift=-6.2mm] (160.8:2) to (160.8:-2);
\draw[gray, xshift=-3mm, yshift=-9.3mm] (160.8:2) to (160.8:-2);
\draw[gray, xshift=-4mm, yshift=-12.4mm] (160.8:2) to (160.8:-2);
\draw[gray, xshift=-5mm, yshift=-15.5mm] (160.8:2) to (160.8:-2);
\draw[<->] (0:1)node[right]{$f_1$} to (0:-1) node[left]{$e_1$};
\draw[<->] (120:1)node[above]{$f_3$} to (120:-1)node[below]{$e_3$};
\draw[<->] (60:1)node[above]{$f_4$} to (60:-1)node[below]{$e_4$};
\draw[<->] (30:{sqrt(3)})node[above]{$f_5$} to (30:-{sqrt(3)}) node[below]{$e_5$};
\draw[<->] (90:{sqrt(3)})node[above]{$f_6$} to (90:-{sqrt(3)}) node[below]{$e_6$};
\draw[<->] (150:{sqrt(3)})node[above]{$ f_2 $} to (150:-{sqrt(3)})node[below]{$e_2$};
\end{tikzpicture}
\end{center}
From Serre's relations we find brackets involving $\g_0$:
 \[
\begin{array}{lllll}
[h_1,e_1]=2e_1 &
[h_1,e_2]=-3e_2 &
[h_2,e_1]= -e_1 &
[h_2,e_2]=2e_2 & [e_1,f_1]=h_1 \\[2mm]
[h_1,f_1]=-2f_1 &
[h_1,f_2]=3f_2 &
[h_2,f_1]=f_1 &
[h_2,f_2]=-2f_2 & [e_2,f_2]=h_2
\end{array}
 \]
The remaining structure relations are written basing on the root
arithmetic with unknown coefficients, which are then uniquely determined
from the Jacobi identity. 

With this knowledge we compute the action of $\fp_+$ on $\m$.
This in turn determines the action of
$\rho=\exp\Bigl(\sum_{k=1}^5s_kf_k\Bigr)\in\exp(\fp_+)$ on
$v=\sum_{k=1}^6v_ke_k$ ($f_6$ acts trivially) as follows:
 \[
\begin{array}{l}
\hspace{-3pt}
\rho(v_1,v_2,v_3,v_4,v_5,v_6)=\\[2mm]
\Bigl(v_1-s_2v_3+(4s_3-2s_1)v_4+(6s_1s_3-2s_1^2)v_5 +
(\tfrac12s_1^2s_2^2-6s_3^2+12s_2s_4)v_6,\\[2mm]
\phantom{1}
v_2+3s_1v_3+6s_1^2v_4+6s_1^3v_5-(6s_1^2s_3+\tfrac32s_1^3s_2+18s_1s_4)v_6,\\[2mm]
\phantom{1}
v_3+4s_1v_4+6s_1^2v_5-(2s_2+6s_3+\tfrac13s_4)v_6,
v_4+3s_1v_5-(3s_3+\tfrac32s_1s_2)v_6, v_5-s_2v_6, v_6\Bigr).
\end{array}
 \]
The group $G_0=\R_\times\times\R_\times$ action on $\g_{-1}$
extends to an automorphism of $\m$.
So we derive the action of $P_{12}=G_0\ltimes\exp(\fp_+)$ on
$\m\setminus\{0\}$ and this yields the decomposition into orbits as follows:

 \begin{itemize}
\item 3 orbits in $T\hat{M}\setminus\Delta^4$: one closed in $\hat{N}$ and two open separated by $\hat{N}$,
\item $\infty$ orbits in $\Delta^4\setminus(\Delta^3\cup H_3)$: there is an absolute invariant in $\Delta^4\setminus\hat{N}$,
\item 3 orbits in $H_3\setminus\Delta^3$: one closed in $\hat{N}$ and two open separated by $\hat{N}$,
\item 2 orbits in $\Delta^3\setminus\Delta^2$: one closed in $H_2$ and one open in the complement,
\item 1 orbit in $\Delta^2\setminus\Delta$,
\item 3 orbits in $\Delta\setminus\{0\}$: two lines and the complement.
 \end{itemize}
Here $H_2=\{a_1\in\Delta^3:h_2(a_1)=0\}$ and
$H_3=\{a_1\in\Delta^4:h_3(a_1)=0\}$
have the defining equations:
 \[
\begin{array}{ll}
\hspace{-3pt}
\,h_2 \,=&\!\! 8p_1r_1-8qr_1-3q_1^2+6rq_1-3r^2,\\[2mm]
h_3   \,=&\!\! 9p^2r_1+9pq_1p_1-9prp_1-9pqq_1+9pqr-18pr_1y_1+4p_1^3-12qp_1^2\\[2mm]
& +12q^2p_1-9p_1q_1y_1+9rp_1y_1-4q^3+9qq_1y_1-9qry_1+9r_1y_1^2.
\end{array}
 \]

We note that $h_2$ is a relative invariant in $\Delta^3$ and
$h_3$ is a relative invariant in $\Delta^4$.
Moreover, $h_3=R_2|_{\Delta^4}$, where $R_2$ is the same relative
invariant as in $G_2/P_1$ (note that the order of $R_2$ drops
to 1 when we restrict to the prolongation of the equation for $\Delta^4$).
Actually, the restriction of $I_2$ to ${\Delta^4}\setminus\hat{N}$ is an absolute differential invariant, where $I_2$ is the second order
differential invariant of generic curves in $G_2/P_1$ pulled back to
$G_2/P_{12}$.

The relative invariant $h_2$ comes from a relative invariant for
integral curves in $G_2/P_2$.

\subsection{Number of invariants}

Similar to what is done in Section \ref{S1},
we compute the Hilbert function $h_k$
counting the number of differential invariants of pure order $k$,
depending on the type $\texttt{t}$ of the orbit of 1-jet,
and tabulate it as follows.

 \[
\begin{array}{|l|ccccccccccl|}
\hline
\quad\texttt{t} \quad\diagdown\quad k
& 0 & 1 & 2 & 3 & 4 & 5 & 6 & 7 & 8 & 9 & \dots \\\hline
T\M\backslash(\Delta^4\cup\hat{N}) \vphantom{\frac{a}b}
& 0 & 0 & 2 & 5 & 5 & 5 & 5 & 5 & 5 & 5 & 5\\
\hat{N}\backslash\Delta^4 & 0 & 0 & 1 & 3 & 4 & 4 & 4 & 4 & 4 & 4 & 4\\
\Delta^4\backslash(\Delta^3\cup\hat{N}\cup H_3) & 0 & 1 & 0 & 3 & 4 & 4 & 4 & 4 & 4 & 4 & 4 \\
H_3\backslash(\Delta^3\cup\hat{N}) & 0 & 0 & 0 & 1 & 3 & 3 & 3 & 3 & 3 & 3 & 3 \\
(\Delta^4\cap\hat{N})\backslash(\Delta^3\cup H_3) & 0 & 0 & 0 & 1 & 3 & 3 & 3 & 3 & 3 & 3 & 3 \\
(\hat{N}\cap H_3)\backslash\Delta^3 & 0 & 0 & 0 & 0 & 0 & 2 & 2 & 2 & 2 & 2 & 2 \\
\Delta^3 \backslash (\Delta^2 \cup H_2) & 0 & 0 & 0 & 1 & 3 & 3 & 3 & 3 & 3 & 3 & 3 \\
H_2\backslash\Delta^2 & 0 & 0 & 0 & 0 & 0 & 2 & 2 & 2 & 2 & 2 & 2 \\
\Delta^2\backslash\Delta & 0 & 0 & 0 & 0 & 0 & 2 & 2 & 2 & 2 & 2 & 2 \\
\Delta\backslash\{0\} & 0 & 0 & 0 & 0 & 0 &0 & 0 & 0 & 0 & 1 & 1\\ \hline
\end{array}
 \]

\medskip

Again we consider in more details only differential invariants
of curves of constant type with either minimal
(actually next to it: integral) or maximal type $\texttt{t}$
of the orbit in 1-jets.

\subsection{Invariants of integral curves}

The first absolute differential invariant $\hat{I}_9$
of curves tangent to $\Delta$ in $\M=G_2/P_{12}$ occurs in order 9.
It coincides with the invariant $I_{10}$
for curves tangent to $\Pi$ in $G_2/P_1$ after the change of coordinates
$r_i=q_{i+1}$ $\forall i$. (Recall that the differential equation for $\Delta$ is obtained from that of $\Pi$ by intersecting with $q_1=r$
and its prolongation.)

An invariant derivation is
 \[
\widehat{\Box}_\imath= \frac{r_1}{\hat{R}_7^{1/6}}\widehat{\mathcal{D}}_x\,,
 \]
where $\hat{R}_7$ coincides with $R_8$ from Section \ref{section_integral_curves1} after the same change of coordinates, and
 \[
\widehat{\D}_x=\p_x+p\p_y+q\p_p+q^2\p_z+r\p_q
+\sum_{i=0}^\infty r_{i+1}\p_{r_i}
 \]
is the operator of total derivative on $\E_\Delta$.

Thus the algebra of differential invariants $\widehat{\mathcal{A}}_\imath$
is generated similarly to Section \ref{section_integral_curves1}, and this
is not surprising: there is a bijection between integral
curves of $\Pi$ and those of $\Delta=\hat\Pi$.
Indeed, the prolongation $\gamma(t)\mapsto(\gamma(t),\dot\gamma(t))$
lifts the integral curves of $\Pi$ to the integral curves of $\Delta$,
and the projection gives the inverse map.

\subsection{Invariants of generic curves}

We have the same relative invariants $R_1$ and $R_2$ as in $G_2/P_1$.
In addition, we have the first order relative invariant
 \[
R_3 = 2pr + 2qp_1 -2ry_1-q^2 -z_1\,,
 \]
that is the pull-back of the contact condition (integral curves) in
$G_2/P_2$.

Some other differential invariants have been computed,
like relative invariant $R_4$ of order 2 (in Maple),
but their formulae are long and we describe the algebra
$\hat{\mathcal{A}}_g$ of invariants differently.

For a generic curve $\hat{\gamma}(t)\subset\M$ its projection
$\gamma(t)=\pi_l\circ\hat{\gamma}(t)\subset M$ is also generic, and hence by
the results of Section \ref{Sec_gen1} possesses a frame
$Y,V,Z,X,W$ along it. A point $\hat\gamma(t)$ over $\gamma(t)$
can be interpreted as a line $\ell_{\gamma(t)}\subset\Pi_{\gamma(t)}$.
There exists a unique $\varrho=\varrho(t)\in\bar\R=\R\cup\infty$
such that $Y+\varrho V\in\ell_{\gamma(t)}$. This $\varrho$ is a function
on the curve, and it defines a rational function on the space of jets
of generic curves in $\M$, denoted by the same symbol.

Let us also note that the differential parameter along the curve
$\hat\gamma$ can be induced from the differential parameter along
its projection $\gamma$. In other words, the invariant derivation
$\Box_g$ from Section \ref{Sec_gen1} induces the following
invariant derivation in $J^\infty(\M,1)$:
 $$
\widehat{\Box}_g=\frac1{\D_x I_2}\cdot\hat{\D}_x,
 $$
where we use the operator of total derivative
 \[
\widehat{\D}_x=\p_x+\sum_{i=0}^\infty \Bigl(y_{i+1}\p_{y_i}+p_{i+1}\p_{p_i}+q_{i+1}\p_{q_i}+z_{i+1}\p_{z_i}
+r_{i+1}\p_{r_i}\Bigr).
 \]

 \begin{theorem}
The algebra $\hat{\mathcal{A}}_g$ of differential invariants of generic
curves in $\M$ is generated by the differential invariants from
Theorem \ref{Th2}, pulled back from $J^\infty(M,1)$ to $J^\infty(\M,1)$,
the invariant $\varrho$ and the derivation $\widehat{\Box}_g$.
 \end{theorem}

 \begin{proof}
A curve $\hat\gamma\subset\M$ is uniquely encoded by its projection
$\gamma\subset M$ and its enhancement $\ell_\gamma$, equivalently
represented by the function $\varrho$. Hence it suffices to add
this invariant to the generating set for $\mathcal{A}_g$ to generate
$\hat{\mathcal{A}}_g$.
 \end{proof}

\section{Invariants of curves in $K^5=G_2/P_2$}

Associated with $P_2$ is the contact gradation of $\g=\op{Lie}(G_2)$
 \begin{equation}\label{g2p2}
\g=\g_{-2}\oplus\g_{-1}\oplus\g_0\oplus\g_1\oplus\g_2
 \end{equation}
with $\dim\g_{\pm1}=4$, $\dim\g_{\pm2}=1$ and $\g_0=\mathfrak{gl}_2$.
The manifold $K=G_2/P_2$ possesses $G_2$-invariant contact structure
$D\subset TK$ and a field of rational normal curves (RNC) in $\mathbb{P}D$,
corresponding to the minimal orbit of (the reductive part of)
the structure group $G_0=GL_2$, also identified with its cone
field $\Gamma\subset D$ (in projectivization we will write $[\Gamma]$).

In coordinates $(x,y,p,q,z)$ on $K$ the contact structure is the
annihilator of $\alpha=dz-p\,dx-q\,dy$ and the rational normal cone
(also abbreviated RNC) is given by the following ideal in $S^\bullet D^*$:
 \[
\langle 3\,dx\,dp-dy\,dq, \sqrt{3}\,dx\,dy-dq^2,
\sqrt{3}\,dp\,dq-dy^2\rangle.
 \]
In other words, RNC is given by the union of 1-parametric family of lines
 \begin{equation}\label{xir}
\Gamma=\bigcup_{r\in\bar\R}\xi_r\subset D,\quad\text{ where }\quad
\xi_r=
\langle(\p_x+p\,\p_z)+r\sqrt{3}\,\p_q+r^2\sqrt{3}\,(\p_y+q\,\p_z)+r^3\p_p\rangle.
 \end{equation}

The tangent to the RNC is the hypersurface in $D$ of degree 4 given by
 \[
T\Gamma=\{a(\p_x+p\,\p_z)+b\,\p_q+c(\p_y+q\,\p_z)+d\p_p\,:\,4\,(ac^3+b^3d)=\sqrt{3}\,(b^2c^2-3a^2d^2+6abcd)\}.
 \]

The curves tangent to those varieties are given by the following equations respectively:
 \begin{gather}
\E_D=\{z_1=p+q\,y_1\},\notag\\
\E_{T\Gamma}=\bigr\{4\,(y_1^3+p_1q_1^3)=\sqrt{3}\,(y_1^2q_1^2-3p_1^2+6y_1p_1q_1),\ z_1=p+q\,y_1\bigl\},\label{EEE}\\
\E_{\Gamma}=\left\{y_1=\frac{q_1^2}{\sqrt{3}},\ p_1=\frac{q_1^3}{3\sqrt{3}},\ z_1=p+\frac{qq_1^2}{\sqrt{3}}\right\}\notag.
 \end{gather}
Note that in our coordinates $(x,y,p,q,z)$ the invariant conformally symplectic structure has the canonical form
$\omega=d\alpha=dx\wedge dp+dy\wedge dq$, while
the RNC has coefficients involving $\sqrt{3}$
(if we normalize RNC standartly,
then the symplectic structure has a coefficient 3).

\subsection{Action and orbits of $P_2$ on 1-jets}

The action of $\fp_+$ on $\m=\g/\fp$ for $\fp=\fp_2$ is nontrivial only on $\g_{-2}$.
Moreover $\g_2$ acts trivially and parametrizing $\g_1$ by the coefficients
$s_1,s_2,s_3,s_4$ in the basis $f_1,f_2,f_3,f_4$ of $\g_1$
dual to the basis $e_1,e_2,e_3,e_4$ of $\g_{-1}$ as described on the picture
 
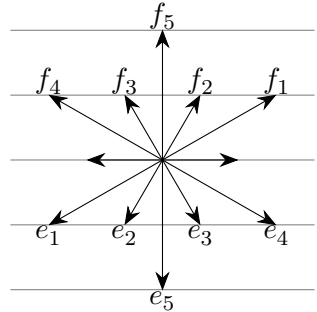
\begin{center}
\begin{tikzpicture}
\draw[gray, yshift=17.2mm] (180:2) to (180:-2);
\draw[gray, yshift=8.6mm] (180:2) to (180:-2);
\draw[gray] (180:2) to (180:-2);
\draw[gray, yshift=-8.6mm] (180:2) to (180:-2);
\draw[gray, yshift=-17.2mm] (180:2) to (180:-2);
\draw[<->] (0:1) to (0:-1);
\draw[<->] (120:1)node[above]{$f_3$} to (120:-1)node[below]{$e_3$};
\draw[<->] (60:1)node[above]{$f_2$} to (60:-1)node[below]{$e_2$};
\draw[<->] (30:{sqrt(3)})node[above]{$f_1$} to (30:-{sqrt(3)}) node[below]{$e_1$};
\draw[<->] (90:{sqrt(3)})node[above]{$f_5$} to (90:-{sqrt(3)}) node[below]{$e_5$};
\draw[<->] (150:{sqrt(3)})node[above]{$f_4$} to (150:-{sqrt(3)})node[below]{$e_4$};
\end{tikzpicture}
 \end{center}
we encode the action as follows:
 \[
(v_1,v_2,v_3,v_4,v_5)\mapsto(v_1+s_1v_5,v_2+s_2v_5,v_3+s_3v_5,v_4+s_4v_5,v_5),
 \]
where $v_i$ are coordinates on $\m=\g_{-1}\oplus\g_{-2}$ associated to the basis $e_1,\dots,e_5$.

The action of $G_0=GL_2$ on $\g_{-1}$ is given by the matrix
 \[
 \left(\begin {array}{cccc}
a^3 & \sqrt{3}\,a^{2}c & \sqrt{3}\,ac^2 & c^3\\ \noalign{\medskip}
\sqrt{3}\,a^2b & a^2d+2\,abc & 2\,acd+bc^2 & \sqrt{3}\,c^2d\\ \noalign{\medskip}
\sqrt{3}\,ab^2 & 2\,abd+b^2c & ad^2+2\,bcd & \sqrt{3}\,cd^2\\ \noalign{\medskip}
b^3 & \sqrt{3}\,b^2d & \sqrt{3}\,bd^2 & d^3\end {array}\right)
 \]
in coordinates $(v_1,v_2,v_3,v_4)$ and it extends to $\g_{-2}$ by $v_5\mapsto(ad-bc)^3v_5$.

Hence the action of $P_2$ decomposes $\m$ into the following orbits
 \begin{itemize}
\item One orbit in $TM\backslash D$,
\item Two orbits in $D\backslash T\Gamma$,
\item One orbit in $T\Gamma\backslash \Gamma$,
\item One orbit in $\Gamma\backslash\{0\}$.
 \end{itemize}

The curves of fixed type $\texttt{t}$ of their 1-jet according to the orbit type
as above, are given by the equations $\E_{\texttt{t}}$ described in \eqref{EEE}.

\subsection{Number of invariants}

Similar to what is done in Section \ref{S1},
we compute the Hilbert function $h_k$
counting the number of differential invariants of pure order $k$,
depending on the type $\texttt{t}$ of the orbit of 1-jet,
and tabulate it as follows.

 \[
\begin{array}{|l|cccccccccccl|}
\hline
\ \texttt{t}\ \diagdown\ k
& 0 & 1 & 2 & 3 & 4 & 5 & 6 & 7 & 8 & 9 & 10 & \dots\\\hline
TM\backslash D & 0 & 0 & 0 & 3 & 4 & 4 & 4 & 4 & 4 & 4 & 4 & 4\\
D\backslash T\Gamma & 0 & 0 & 0 & 1 & 2 & 3 & 3 & 3 & 3 & 3 & 3 & 3\\
T\Gamma \backslash \Gamma & 0 & 0 & 0 & 0 & 0 & 1 & 2 & 2 & 2 & 2 & 2 & 2\\
\Gamma \backslash \left\{ 0 \right\}   & 0 & 0 & 0 & 0 & 0 & 0 & 0 & 0 & 0 & 0 & 1 & 1 \\\hline
\end{array}
 \]

\medskip

Again we consider in more details only differential invariants
of curves of constant type with either minimal or maximal type $\texttt{t}$
of the orbit in 1-jets.

\subsection{Invariants of minimal integral curves}

There are several types $\texttt{t}$ of integral curves for $G_2/P_2$,
we consider those that are tangent to RNC $\Gamma$. The algebra
of differential invariants $\mathcal{I}_\imath$ of these curves is generated by
 \[
\widetilde{I}_{10}=\frac{\widetilde{R}_{10}}{\widetilde{R}_8^{7/3}}\quad
\text{ and }\quad
\widetilde\Box=\frac{q_2}{\widetilde R_8^{1/6}}\mathcal{D}_x\,,
 \]
where $\D_x$ is the operator of total derivative restricted to $\E_\Gamma$,
 \[
\begin{array}{ll}
\widetilde{R}_8 \,\,= \!\! & 196\,q_8q_2^5-3136\,q_2^4q_3q_7-5936\,q_2^4q_4q_6
-3605\,q_2^4q_5^2+26208\,q_2^3q_3^2q_6+83538\,q_2^3q_3q_4q_5\\[2mm]
& +18252\,q_2^3q_4^3-144144\,q_2^2q_3^3q_5-281853\,q_2^2q_3^2q_4^2
+555984\,q_3^4q_4q_2-247104\,q_3^6
\end{array}
 \]
corresponds to the invariant $R_8$ from the case $G_2/P_1$, and similar
for $\widetilde{R}_{10}$ (see Theorem \ref{thm1} and Remark \ref{rk1}).
Explicitly this correspondence will be explained in the next section.

\subsection{Invariants of generic curves}\label{gencurveG2P2}

For generic curves $\gamma\subset K$ transversal to the distribution $D$,
the first relative invariant appears in order 1 and it corresponds to
tangency with $D$:
 \[
\widetilde{R}_1 = z_1-p-qy_1.
 \]
The second order relative invariant $\widetilde{R}_2$ has more complicated
formula, but the most difficult are absolute differential invariants
$\widetilde{I}_{3a}$, $\widetilde{I}_{3b}$, $\widetilde{I}_{3c}$,
$\widetilde{I}_{4d}$ of orders $3\,(\times3)$ and 4, which together with
the invariant derivation $\widetilde{\Box}_g$ generate the algebra
of differential invariants $\mathcal{I}_g$.
Below we explain 
how to obtain an invariant frame that, in principle, determines all
basic invariants.

It is known \cite{CS} that for every contact parabolic geometry,
in particular for $K=G_2/P_2$, there is a unique
(up to projective reparametrization) canonical (distinguished) curve
through any point $a\in K$ in any direction $v\in T_aK\setminus D_a$.
For $\gamma(0)=a$ choose $X=\dot\gamma(0)$ and denote this curve
by $\delta_X$. We treat both curves $\gamma$ and $\delta_X$ as
unparametrized. They have the same 1-jet, and therefore
their difference canonically determines 2-jet
$\zeta_X\in S^2T^*_a\gamma\otimes\nu_a$, where $\nu_a=T_aK/\langle X\rangle$
is the normal to $\gamma$ at $a$. The image
$\zeta_X(X,X)\in T_aK\,\op{mod}X$ defines uniquely a 2-plane in $T_aK$,
containing $X$, the intersection of which with $D_a$ is a line
denoted $\Upsilon_X$.

Next we use the projective geometry of $\mathbb{P}D$ equipped with RNC
$[\Gamma]$. The above construction gives a point
$[\Upsilon_X]\in\mathbb{P}D$. We assume the genericity condition
$[\Upsilon_X]\not\in T[\Gamma]$. In this case there exists
a unique bisection $L$ of $[\Gamma]$ containing $[\Upsilon_X]$,
intersecting the RNC in two points $\lambda_X^\pm$
(there is no canonical way to distinguish between $\pm$, so these points
enter non-numerated; in the case $[\Upsilon_X]\in T[\Gamma]\setminus[\Gamma]$ they coincide and in the case $[\Upsilon_X]\in[\Gamma]$ there are
infinitely many lines $L$).

Moreover we can introduce two more points
$\mu_X^\pm=T_{\lambda_X^\pm}[\Gamma]\cap T^2_{\lambda_X^\mp}[\Gamma]$
on the (unique) intersection of the first and second tangents at
the points $\lambda_X^+$ and $\lambda_X^-$ (or interchange).
The corresponding lines in $D$ can be normalized so that to form
a conformally symplectic basis with respect to $\omega$.
Only one overall scale is missing
to obtain the frame from those vectors jointly with $X$
and to fix the contact form. This can be normalized via a
differential invariant.

The formulae are rather complicated,
so instead we show in the next section how to relate
the equivalence problem for generic curves in $K=G_2/P_2$
to those in $M=G_2/P_1$.

\section{Twistor correspondence}

The three realizations of $G_2$, acting on various $G_2/P$ as discussed above,
are conveniently related by the following double fibration
(parabolic subgroups $P$ correspond to crosses on the Dynkin-Satake diagrams).

 
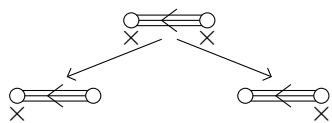
\begin{figure}[h]
 \begin{center}
 \begin{tikzpicture}
 \draw (0,0) -- (1,0);
 \draw (0,0.07) -- (1,0.07);
 \draw (0,-0.07) -- (1,-0.07);
 \draw (0.6,0.15) -- (0.4,0) -- (0.6,-0.15);
 \node[draw,circle,inner sep=2pt,fill=white] at (0,0) {};
 \node[draw,circle,inner sep=2pt,fill=white] at (1,0) {};
 \node[below] at (0,-0.1) {$\times$};
 \draw (3,0) -- (4,0);
 \draw (3,0.07) -- (4,0.07);
 \draw (3,-0.07) -- (4,-0.07);
 \draw (3.6,0.15) -- (3.4,0) -- (3.6,-0.15);
 \node[draw,circle,inner sep=2pt,fill=white] at (3,0) {};
 \node[draw,circle,inner sep=2pt,fill=white] at (4,0) {};
 \node[below] at (4,-0.1) {$\times$};
 \draw (1.5,1) -- (2.5,1);
 \draw (1.5,1.07) -- (2.5,1.07);
 \draw (1.5,0.93) -- (2.5,0.93);
 \draw (2.1,1.15) -- (1.9,1) -- (2.1,0.85);
 \node[draw,circle,inner sep=2pt,fill=white] at (1.5,1) {};
 \node[draw,circle,inner sep=2pt,fill=white] at (2.5,1) {};
 \node[below] at (1.5,0.9) {$\times$};
 \node[below] at (2.5,0.9) {$\times$};
\path[->,>=angle 90](1.9,0.75) edge (0.65,0.25);
\path[->,>=angle 90](2.1,0.75) edge (3.35,0.25);
 \end{tikzpicture}
\end{center}
\end{figure}

The arrows are projections corresponding to the inclusions $P_1\hookleftarrow P_{12}\hookrightarrow P_2$.
Below we explain how this correspondence relates three equivalence problems studied in this paper.

\subsection{Correspondence for points}

For $G_2/P_{12}$ we used the nomenclature $\M$ because it was the geometric
prolongation of $M=G_2/P_1$. In coordinate language
the affine chart $\R^5(x,y,p,q,z)$ of $M$ is covered by the affine chart
$\R^6(x,y,p,q,z,r)$, where $r$ is such that the point of $\M$
is represented by the line
$\ell=\langle\p_x+p\p_y+q\p_p+q^2\p_z+r\p_q\rangle$ in the
distribution \eqref{Pi}.
The rank 2 distribution $\Delta$ of $G_2/P_{12}$ with the canonical split
into $1+1$ line subbundles is given by \eqref{PPi}.

We can also represent $G_2/P_{12}$ as the geometric prolongation of $K=G_2/P_2$, so that its points $\hat{b}=(b,\xi_r)$ are lines in RNC
$\xi_r\subset D_b$, $b\in K$, see \eqref{xir}. In coordinates
this gives $G_2/P_{12}=\hat{K}$ with affine chart $\R^6(x,y,p,q,z,r)$,
and this is naturally equipped with the rank 2 distribution
of growth $(2,3,4,5,6)$ that is canonically split into $1+1$ lines
subbundles as follows:
 \begin{equation}\label{PCi}
\tilde\Delta=\langle (\p_x+p\,\p_z)+r\sqrt{3}\,\p_q+r^2\sqrt{3}\,(\p_y+q\,\p_z)+r^3\p_p,\p_r\rangle.
 \end{equation}

There is a diffeomorphism $\varphi:(\M,\Delta)\to(\hat{K},\tilde{\Delta})$
that interchanges the first and the second generators of the distributions.
In other words, the vertical line (fiber to $\pi_l$) in $\Delta$
is mapped to the horizontal line in $\tilde\Delta$ and the horizontal line
in $\Delta$ is mapped to the vertical line (fiber to $\pi_r$) in $\tilde\Delta$. This fits the following diagram

 \begin{center}
\begin{tikzpicture}
\node () at (3,2.5) {}; \node () at (0,-0.3) {};
\node (A) at (2,2) {$G_2/P_{12}\,:\,(\M,\Delta)$};
\node (B) at (5,2) {$(\hat{K},\tilde\Delta)$};
\node (C) at (0,0) {$G_2/P_1\,:\,(M,\Delta)$};
\node (D) at (6.5,0) {$G_2/P_2\,:\,(K,D,\Gamma)$};
 \path[->,font=\scriptsize,>=angle 90]
(2.8,1.7) edge node[above] {$\pi_l$\hphantom{1em}} (0.8,0.3)
(3.6,2) edge node[above]{${}^{\displaystyle\varphi}$} (4.3,2)
(5.1,1.7) edge node[above] {$\hphantom{1em}\pi_r$} (7.1,0.3);
\end{tikzpicture}
 \end{center}
where $\pi_l(x,y,p,q,z,r)=(x,y,p,q,z)$ and
$\pi_r(x,y,p,q,z,r)=(x,y,p,q,z)$ in the corresponding coordinates.
The required transformation is given by formula
 \begin{multline*}
\varphi(x,y,p,q,z,r)=\\
\left( -\frac{1}{r},\, \sqrt{3}\Bigl(2p-\frac{q^2}{r}\,\Bigr),\,
3z-\frac{q^3}{r},\, \sqrt{3}\Bigl(x-\frac{q}{r}\Bigr),\,
6(xp-y)-\frac{3}{r}\Bigl(z+xq^2\Bigr)+\frac{2q^3}{r^2},\, q \right).
 \end{multline*}

\subsection{Correspondence for integral curves}

Integral curves for $(M,\Pi)$ are given by equation \eqref{IC1}
and their prolongations to $(\M,\Delta)$ are determined by the
additional constraint $r=q_1$. 
Thus there is a 1:1 correspondence between integral curves of
$(M,\Pi)$ and integral curves of $(\M,\Delta)$. The invariant constraint
$r_1=0$ (or $r=\op{const}$) determines a 1-parametric family of
integral curves through any point called abnormal extremals for $\Pi$.
They are projections of the integral curves for the horizontal
line distribution given by the first generator of \eqref{PPi}.

Minimal integral curves for $(K,D,\Gamma)$ are given by equation
$\E_\Gamma$ of \eqref{EEE} and their prolongation to $(\hat{K},\tilde\Delta)$ are determined by the additional constraint $r=q_1/\sqrt{3}$.
Thus there is a 1:1 correspondence between integral curves of
$(K,D,\Gamma)$ and integral curves of $(\hat{K},\tilde\Delta)$.
The invariant constraint $r_1=0$ (or $r=\op{const}$) determines straight line generators of the RNC through any point of $K$.
They are projections of the integral curves for the horizontal
line distribution given by the first generator of \eqref{PCi}.

This correspondence on the level of jets is summarized in the following
diagram, where we denote by $\jmath_{l,r}$ the lifts defined above, they
are right inverse to the projections $\pi_{l,r}$.

 \begin{center}
\begin{tikzpicture}
\node () at (3,2) {}; \node () at (0,-0.1) {};
\node (A) at (2.7,1.5) {$\E_\Delta$};
\node (C) at (0.1,0) {$\E_\Pi$};
\node (D) at (5.2,0) {$\E_\Gamma$};
 \path[->,font=\scriptsize,>=angle 90]
(2.2,1.2) edge node[above] {$\pi_l$\hphantom{1em}} (0.2,0.3)
(3.2,1.2) edge node[above] {$\hphantom{1em}\pi_r$} (5.2,0.3)
(0.5,0.3) edge node[below] {\hphantom{1em}$\jmath_l$} (2.5,1.2)
(4.9,0.3) edge node[below] {$\jmath_r\hphantom{1em}$} (2.9,1.2);
\end{tikzpicture}
 \end{center}

Note that $\E_\Pi\times\mathbb{P}^1\simeq\E_\Delta\simeq
\E_\Gamma\times\mathbb{P}^1$ (since any integral curve is uniquely
lifted given a point in the fiber) and $\E_\Pi\simeq\E_\Gamma$, which
explains isomorphism of the algebras of differential invariants.

\subsection{Correspondence for generic curves}

The above correspondence cannot be extended to all curves,
however we can produce lifts for generic curves.

On the left side of the double (twistor) fibration the lift
is determined from the observation of Section \ref{Sec_gen1} that
$X$ given by \eqref{XX} determines $Y\in\Pi$ up to scale
by the condition $g(X,Y)=0$. Setting $Y=\p_x+p\p_y+q\p_p+q^2\p_z+r\p_q$
this and \eqref{conf} gives the formula for $r$, from which we conclude
that the lift is given by the following formula (and its prolongations):
 \[
\jmath_l(x,y,p,q,z)=\left(x,y,p,q,z,\frac{2qp_1-z_1-q^2}{2(y_1-p)}\right).
 \]

On the right side of the double fibration the computation is a bit more
involved. First we derive the formula for distinguished curves in
direction $\g_{-2}$. By the mentioned general result \cite[\S5.3.7]{CS},
there is a unique unparametrized distingushed curve of that type
in any non-contact direction on $TK$. This gives an invariant section
$J^1(K,1)\dashrightarrow J^2(K,1)$ defined on a Zariski open set.
(An alternative way to check it: the stabilizer of a generic
$b_1\in J^1$ is $GL_2\times\R_\times$ that acts on $\pi_{2,1}^{-1}(b_1)\simeq\R^4$
via an irreducible representation of $GL_2$ that has one fixed point.)

The explicit formula involves matrix realization $G_2\subset SO(3,4)$
described in \cite{HS} on the level of Lie algebras;
the reference specifies the $\fp_1$ grading but one can
also identify $\fp_2$. The corresponding parabolic subgroup $P_2$
can be coordinized via $GL_2\ltimes\exp(\fp_+)$ and the action of this
on $\m=\g_-$ is then explicitly derived. The formulae
for the distinguished curves (omitted here, see Maple's supplement)
imply the formula for the above section, or equivalently
for a point $[\Upsilon_X]\in\mathbb{P}D$ as defined in Section
\eqref{gencurveG2P2}:
 \begin{multline*}
\Upsilon_X=
\frac{2\sqrt{3}\,q_1^3-9\,qy_2-18\,y_1q_1+9\,z_2}{9(qy_1+p-z_1)}\,
(\partial_x + p\partial_z)+
\left(q_2-\frac{2\,\sqrt{3}\,y_1^2+3\,qq_1y_2+2\,y_1q_1^2-3\,q_1z_2}
{3(qy_1+p-z_1)}\right)\,\partial_q\\
+\left(y_2+\frac{2\,\sqrt{3}\,p_1q_1^2-3\,qy_1y_2-4\,y_1^2q_1-6\,y_1p_1+3\,y_1z_2}
{3(qy_1+p-z_1)}\right)\,(\partial_y+q\partial_z)\\
+\left(p_2-\frac{2\,\sqrt{3}\,y_1^3+9\,qp_1y_2+18\,p_1^2-9\,p_1z_2}
{9(qy_1+p-z_1)}\right)\,\partial_p.
 \end{multline*}

Next, given a point $[a:b:c:d]\in\mathbb{P}D$ the RNC secant line
through it intersects $[\Gamma]$ at the points
corresponding to the parameter $r$ from \eqref{xir} so:
 \[
\lambda_X^\pm =
\frac{3\,ad-bc \pm \sqrt{(3\,ad-bc)^2-4(\sqrt{3}ac-b^2)(\sqrt{3}bd-c^2)}}
{2(\sqrt{3}ac-b^2)}\,.
 \]
Equivalently, the points $r=\lambda_X^\pm$ are the solutions of the
quadratic equation
 \[
(\sqrt{3}\,ac -b^2)r^2+(bc-3\,ad)\,r+\sqrt{3}\,bd-c^2=0.
 \]
This formula for $r$ composed with the formula for $\Upsilon_X$
(the coefficients $a,b,c,d$ are extracted in the order of appearance)
defines two lifts of generic curves from $K$ to $\hat{K}\simeq\M$:
 \[
\jmath^\pm_r(x,y,p,q,z)=(x,y,p,q,z,\lambda_X^\pm\circ[\Upsilon_X]).
 \]

This correspondence on the level of jets is summarized in the following
diagram:
 \begin{center}
\begin{tikzpicture}
\node () at (3,2) {}; \node () at (0,-0.1) {};
\node (A) at (2.5,1.5) {$J^\infty(\M,1)$};
\node (C) at (0,0) {$J^\infty(M,1)$};
\node (D) at (5,0) {$J^\infty(K,1)$};
 \path[->,font=\scriptsize,>=angle 90]
(2.2,1.2) edge node[above] {$\pi_l$\hphantom{1em}} (0.2,0.3)
(3.2,1.2) edge node[above] {$\hphantom{1em}\pi_r$} (5.2,0.3);
 \path[dashed,->,font=\scriptsize,>=angle 90]
(0.5,0.3) edge node[below] {\hphantom{1em}$\jmath_l$} (2.5,1.2)
(4.9,0.3) edge node[below] {$\jmath_r\!\!{}^\pm\hphantom{1em}$} (2.9,1.2);
\end{tikzpicture}
 \end{center}
The dashed arrows are defined on open dense subsets of their domains,
are right inverse to the corresponding projections, and in addition,
$\jmath_r\!\!{}^\pm$ is 1:2 map. This can be seen as an analog of
the B\"acklund transformation, so that for one (jet of) curve in
$K$ we obtain two such in $M$. This allows to derive the algebra of
differential invariants $\mathcal{I}_g$ of curves in $G_2/P_2$
from the results of Section \ref{S1} by averaging the invariants
thereof on the two branches $\pi_l\circ\jmath_r^\pm$.

\section{Concluding remarks}

The computations in this paper demonstrate the method of
differential invariants for $G_2$ action on curves in generalized
flag varieties.
The group is more complicated than the projective group $PSL_{n+1}$
mentioned in the introduction, and we address the corresponding challenges.

For (minimal) integral curves the approach is very effective and provides
a complete description of the algebra. This has to be compared with
the method of moving frame \cite{C}; a moving frame for this problem
was constructed in \cite{DZ} but the algebra of invariants was not derived.

A modification of this method, the equivariant moving frame \cite{O},
is not applicable as it relies on an explicit Lie group
parametrization, which is non-trivial for $G_2$ (one has to resolve
the quadratic and cubic equations defining the group). We worked mainly
with the Lie algebra. Even in this case for generic curves the direct
computations fail, and we had to evoke geometric arguments to arrive
to the basic invariants, in particular exploiting the ideas of moving frames.

The results of this paper concern only curves in homogeneous flag varieties
$G_2/P$, but they can be extended to more general case of curves
in curved $M^5$ of type $(G_2,P_1)$ etc. Indeed in this more general case
the symmetry algebra of such $M^5$ is smaller than $G_2$ yet the invariants
can be found by the same method. In particular, the stratification
of 1-jets makes a perfect sense in the curved case and one can derive
relative invariants similar to $R_1,R_2$ in Section \ref{S1}
(for $R_1$ this is straightforward) leading to
absolute differential invariants.

The invariance is meant here in the following sense: If $\phi:M_1\to M_2$
is an equivalence between two different spaces with their (2,3,5)
distributions, sending one curve $\gamma_1\subset M_1$ to another
$\gamma_2\subset M_2$ then the invariants are superposed.
Since the structural group $P_1$ for the Cartan bundle associated to
this normal parabolic geometry \cite{CS} was central in our computations,
the basic invariants are expected to generalize.

For the geometry of type $(G_2,P_{12})$ the situation is completely
similar because it is functorially equivalent to the geometry of
type $(G_2,P_1)$. However in the curves case $(G_2,P_2)$ type geometry
fails the twistor correspondence, so this would require a separate
consideration.

\end{document}